\documentclass{amsart}
\RequirePackage{fix-cm}
\RequirePackage[l2tabu,orthodox]{nag}
\RequirePackage[all,warning]{onlyamsmath}
\usepackage{latexsym}
\usepackage{amsfonts,amssymb,amsmath,amsthm}
\usepackage{graphicx}
\usepackage{braket}
\usepackage{mleftright}
\usepackage{hyperref}
\usepackage{cleveref}
\usepackage{enumerate}
\numberwithin{equation}{section}

\theoremstyle{plain}
\newtheorem{theorem}{Theorem}[section]
\newtheorem{corollary}[theorem]{Corollary}
\newtheorem{lemma}[theorem]{Lemma}
\newtheorem{proposition}[theorem]{Proposition}

\theoremstyle{definition}

\newtheorem{example}[theorem]{Example}
\newtheorem{remark}[theorem]{Remark}

\crefname{section}{Section}{Sections}
\crefname{appendix}{Appendix}{Appendices}
\crefname{theorem}{Theorem}{Theorems}
\crefname{lemma}{Lemma}{Lemmas}
\crefname{corollary}{Corollary}	{Corollaries}
\crefname{proposition}{Proposition}{Propositions}
\crefname{claim}{Claim}{Claims}
\crefname{conjecture}{Conjecture}{Conjectures}
\crefname{definition}{Definition}{Definitions}
\crefname{problem}{Problem}{Problems}
\crefname{example}{Example}{Examples}
\crefname{remark}{Remark}{Remarks}
\crefname{figure}{Figure}{Figures}
\crefname{equation}{}{}
\crefname{enumi}{}{}
\crefformat{enumi}{(#2#1#3)}
\Crefformat{enumi}{(#2#1#3)}

\newcommand{\QED}{\hfill \ensuremath{\Box}}
\newcommand{\R}{\mathbb{R}}
\newcommand{\abs}[1]{\mleft| #1 \mright|}
\newcommand{\norm}[1]{\mleft\| #1 \mright\|}
\newcommand{\prn}[1]{\mleft( #1 \mright)}
\newcommand{\ang}[1]{\mleft\langle #1 \mright\rangle}
\newcommand{\ld}{,\ldots,}
\newcommand{\ep}{\varepsilon}
\newcommand{\wt}{\widetilde}

\DeclareMathOperator*{\minimize}{minimize}
\DeclareMathOperator{\M}{Min}
\DeclareMathOperator{\WM}{WMin}
\DeclareMathOperator{\E}{S}
\DeclareMathOperator{\WE}{WS}
\DeclareMathOperator{\SE}{SS}
\DeclareMathOperator{\FDH}{FDH}

\title[Equality of weak efficiency and efficiency]{Characterization of the equality\\
of weak efficiency and efficiency\\
on convex free disposal hulls}

\author[N.~Hamada]{Naoki Hamada}
\address{
    KLab Inc.,
    Roppongi Hills Mori Tower,
    6-10-1 Roppongi, Minato-ku, Tokyo, 106-6122, Japan;
    RIKEN AIP-FUJITSU Collaboration Center,
    The RIKEN Center for Advanced Intelligence Project,
    Nihonbashi 1-chome Mitsui Building, 15th floor, 1-4-1 Nihonbashi, Chuo-ku, Tokyo 103-0027, Japan
}
\email{hamada-n@klab.com}

\author{Shunsuke \textsc{Ichiki}}
\address{
Department of Mathematical and Computing Science,
School of Computing,
Tokyo Institute of Technology,
Tokyo 152-8552,
Japan}
\email{ichiki@c.titech.ac.jp}

\begin{document}

\maketitle

\begin{abstract}
In solving a multi-objective optimization problem by scalarization techniques, solutions to a scalarized problem are, in general, weakly efficient rather than efficient to the original problem.
Thus, it is crucial to understand what problem ensures that all weakly efficient solutions are efficient.
In this paper, we give a characterization of the equality of the weakly efficient set and the efficient set, provided that the free disposal hull of the domain is convex.
By using this characterization, we obtain various mathematical applications.
As a practical application, we show that all weakly efficient solutions to a multi-objective LASSO with mild modification are efficient.
\end{abstract}

%%%%%%%%%%%%%%%%%%%%%%%%%%%%%%%%%%%%%%%%%%%%%%%%%%%%%%%%%%%%%%%%%%%%%%%%%%%%%%%%
\section{Introduction}\label{sec:intro}
The aim of multi-objective optimization is to find efficient solutions to a given problem.
In order to do so, various scalarization techniques have been developed so far (see for example \cite{Pascoletti1984,Das1998,Miettinen1999,Messac2004,Zhang2007,Eichfelder2008,Sato2014}).
Nevertheless, there is no scalarization method that ensures for a wide variety of problems that all solutions optimal to scalarized problems are efficient to the original problem.
In general, scalarization methods only ensure that their solutions are \emph{weakly} efficient to the original problem, which means users may waste computation resources for finding inefficient, undesirable solutions.
Thus, it is crucial to understand conditions that the weak efficiency coincides with the efficiency.

In the literature, the relationship between the weak efficiency and the efficiency has been investigated.
In some cases, the set of weakly efficient solutions to a given problem can be described as the union of the sets of efficient solutions to its subproblems \cite{Lowe1984,Ward1989,Malivert1994,Benoist2000}.
This property was named the \emph{Pareto reducibility} \cite{Popovici2005} and further investigated \cite{Popovici2006,Popovici2008,LaTorre2010}.
Some relationships of the weak efficiency and the efficiency on quasi-convex problems are collected in \cite{Luc1989,Ehrgott2002}.
However, the equality between the weak efficiency and the efficiency, both of which are of the \emph{original} problem (rather than subproblems), is still unclear.

In this paper, we give a characterization of the equality of the set of weakly efficient solutions and the set of efficient solutions, provided that the free disposal hull \cite{Al-Mezel2014} of the image of an objective mapping is convex
(see \cref{thm:main_app} in \cref{sec:main-results}).
This claim is derived from our main theorem (\Cref{thm:main} in \cref{sec:main-results}), which gives a similar characterization of the equality of the weakly efficient set and the efficient set on a partially ordered Euclidean space without objective functions.
Furthermore, \cref{thm:main_app} yields various mathematical applications (see \cref{thm:app_imi,thm:app_convex_mapping,thm:app_convex,thm:app_strong} in \cref{sec:application}).
Moreover, as a practical application of \cref{thm:main_app}, we show that all weakly efficient solutions to a multi-objective LASSO with mild modification are efficient.

This paper is organized as follows.
First, in \cref{sec:main-results}, we present the main results (\cref{thm:main,thm:main_app}).
Implications of \cref{thm:main} are discussed with illustrative examples in \Cref{sec:example}.
\cref{sec:main} is devoted to the proof of \cref{thm:main}.
In order to state and prove \cref{thm:app_imi,thm:app_convex_mapping,thm:app_convex,thm:app_strong} in \cref{sec:application}, we prepare some definitions and lemmas in \cref{sec:lemma}.
In \cref{sec:practical-applications}, we investigate a multi-objective version of the LASSO with mild modification as a practical application of our result.
\Cref{sec:conclusion} provides concluding remarks.

%%%%%%%%%%%%%%%%%%%%%%%%%%%%%%%%%%%%%%%%%%%%%%%%%%%%%%%%%%%%%%%%%%%%%%%%%%%%%%%%
\section{Preliminaries and the statements of the main results}\label{sec:main-results}
Unless otherwise stated, it is not necessary to assume that mappings are continuous.
Throughout this paper, we set 
\begin{align*}
    M=\set{1\ld m},
\end{align*}
where $m$ is a positive integer.
We denote a nonempty subset of $M$ by $I$.
Let $y = (y_1 \ld y_m)$ and $y' = (y_1' \ld y_m')$ be two elements of $\R^m$.
The inequality $y \leq_I y'$ (resp., $y <_I y'$) means that $y_i \leq y_i'$ (resp., $y_i < y_i'$) for all $i \in I$.
The inequality $y \lneq_I y'$ means that $y_i \leq y_i'$ for all $i \in I$ and there exists $j\in I$ such that $y_j<y'_j$.

Let $Y$ be a subset of $\R^m$.
Let $\M_I Y$ (resp., $\WM_I Y$) be the set consisting of all elements $y' \in Y$ such that there does not exist any element $y \in Y$ satisfying $y \lneq_I y'$ (resp., $y <_I y'$).
For simplicity, set $\M Y = \M_M Y$ and $\WM Y = \WM_M Y$.
Then, the set $\M Y$ (resp., $\WM Y$) is called the \emph{efficient set} (resp., the \emph{weakly efficient set}) of $Y$.

For a subset $Z$ of $\R^m$, the set $Z + \R^m_{\geq 0}$ is called the \emph{free disposal hull} of $Z$ (denoted by $\FDH Z$), where
\begin{equation*}
    \R^m_{\geq 0} = \Set{(y_1 \ld y_m) \in \R^m | y_1 \geq 0 \ld y_m \geq 0}.
\end{equation*}
For details on free disposal hulls, see \cite{Al-Mezel2014}.
A subset $Z$ of $\R^m$ is said to be \textit{convex} if $tx + (1 - t)y \in Z$ for all $x, y \in Z$ and all $t \in [0, 1]$.

The main theorem of this paper is the following.
\begin{theorem}\label{thm:main}
    Let $Y$ be a subset of $\R^m$.
    If the free disposal hull of $Y$ is convex, then the following $(\alpha)$ and $(\beta)$ are equivalent:
    \begin{enumerate}
        \item[$(\alpha)$] $\WM Y = \M Y$.
        \item[$(\beta)$] $\displaystyle \bigcup_{\emptyset \neq I \subseteq M} \M_I Y \subseteq \M Y$.
    \end{enumerate}
\end{theorem}

\begin{remark}\label{remark:1}
    As in the proof of \cref{thm:main}, the hypothesis that the free disposal hull of $Y$ is convex is used only in the proof of $(\beta)\Rightarrow (\alpha)$ (see \cref{subsec:main_proof_2}).
    In the proof of $(\alpha)\Rightarrow (\beta)$ of \cref{thm:main}, it is not necessary to assume that the free disposal hull of $Y$ is convex (see \cref{subsec:main_proof_1}).
\end{remark}

Now, in order to state \cref{thm:main_app}, we will prepare some definitions.
Let $f = (f_1 \ld f_m): X \to \R^m$ be a mapping, and $I = \Set{i_1 \ld i_k}$ $(i_1 < \cdots < i_k)$ be a nonempty subset of $M$, where $X$ is a given set and $k$ is the number of the elements of $I$.
Let $f_I: X \to \R^k$ be the mapping defined by $f_I = (f_{i_1} \ld f_{i_k})$.
A point $x^* \in X$ is called an \emph{efficient solution} (resp., a \emph{weakly efficient solution}) to the following multi-objective optimization problem:
\begin{align*}
    \minimize_{x \in X} f_I(x) = (f_{i_1}(x) \ld f_{i_k}(x)),\label{eq:mop}
\end{align*}
if $f(x^*) \in \M_I f(X)$ (resp., $f(x^*) \in \WM_I f(X)$).
By $\E(f_I, X)$ (resp., $\WE(f_I, X)$), we denote the set consisting of all efficient solutions (resp., all weakly efficient solutions).
Namely,
\begin{align*}
    \E(f_I, X)  & = f^{-1}(\M_I f(X)),\\
    \WE(f_I, X) & = f^{-1}(\WM_I f(X)).
\end{align*}

It is well known that a solution to a weighting problem is a weekly efficient solution (for example, see \cite[Theorem~3.1.1 (p.~78)]{Miettinen1999}).
On the other hand, a solution to a weighting problem is not necessarily an efficient solution.
For a given mapping $f: X \to \R^m$, if $\WE(f, X) = \E(f, X)$, then a solution to weighting problem is always an efficient solution.
Therefore, characterizations of $\WE(f, X) = \E(f, X)$ are useful and significant.

As an application of \cref{thm:main} to multi-objective optimization problems, we have the following, which can be easily shown by setting $Y = f(X)$ in \cref{thm:main}.
\begin{proposition}\label{thm:main_app}
    Let $f = (f_1 \ld f_m): X \to \R^m$ be a mapping, where $X$ is a given set.
    If the free disposal hull of $f(X)$ is convex, then the following $(\alpha)$ and $(\beta)$ are equivalent:
    \begin{enumerate}
        \item[$(\alpha)$] $\WE(f, X) = \E(f, X)$.
        \item[$(\beta)$] ${\displaystyle \bigcup_{\emptyset \neq I \subseteq M}\E(f_I, X) \subseteq \E(f, X)}$.
    \end{enumerate}
\end{proposition}

%%%%%%%%%%%%%%%%%%%%%%%%%%%%%%%%%%%%%%%%%%%%%%%%%%%%%%%%%%%%%%%%%%%%%%%%%%%%%%%%
\section{Illustration of \texorpdfstring{\cref{thm:main}}{Theorem 2.1}}\label{sec:example}
In this section, we denote by $\overline{x y}$ the line segment with end points $x, y\in \R^2$.
First, we see an example that $(\beta)$ implies $(\alpha)$ in \cref{thm:main}.
\begin{example}\label{ex:1}
    Let us consider the situation shown in \cref{fig:example1}.
    The domain $Y$ in this case is defined by the convex hull of four points $p_1 = (0, 1)$, $p_2 = (1, 0)$, $p_3 = (2, 1)$, $p_4 = (1, 2)$, as shown in dark gray in the figure.
    
    \begin{figure}[ht]
        \includegraphics[width=.66\textwidth]{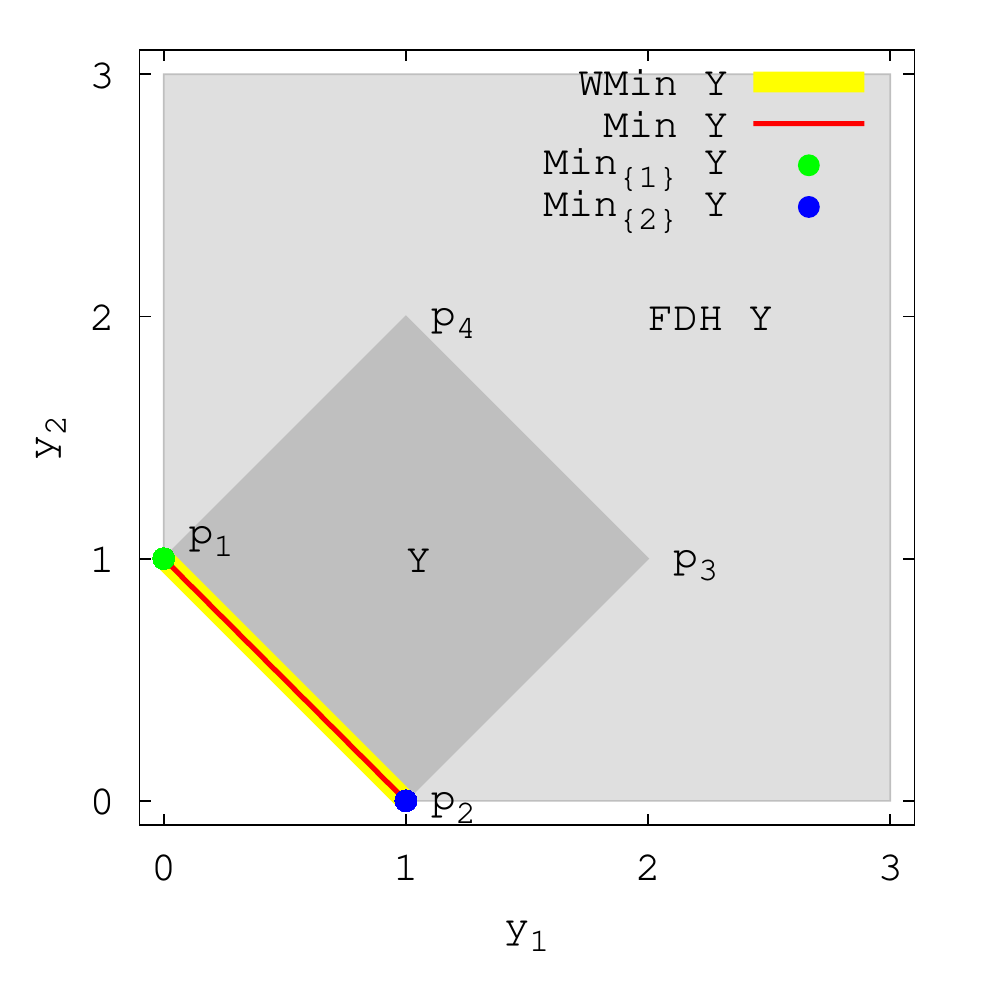}%
        \caption{The condition $(\beta)$ implies $(\alpha)$ on convex $\FDH Y$.}\label{fig:example1}
    \end{figure}
    
    It is easy to check that
    \begin{align*}
        \M_{\Set{1}} Y & = \Set{p_1},\\
        \M_{\Set{2}} Y & = \Set{p_2},\\
        \WM Y = \M   Y & = \overline{p_1 p_2}.
    \end{align*}
    Since $\M_{\Set{1}} Y \subseteq \M Y$ and $\M_{\Set{2}} Y \subseteq \M Y$, we can see the condition $(\beta)$ in \cref{thm:main} holds.
    The free disposal hull of $Y$ is the region shown in light gray in the figure, which is a convex set.
    Thus, we can apply \cref{thm:main} and obtain $(\alpha)$.
    Actually, the condition $\WM Y = \M Y$ holds in this example.
\end{example}

On the other hand, \cref{ex:2} shows that $\FDH Y$ is convex, but both $(\alpha)$ and $(\beta)$ do not hold.
\begin{example}\label{ex:2}
    Let us consider the situation shown in \cref{fig:example2}.
    The domain $Y$ is the convex hull of four points $p_1 = (0, 2)$, $p_2 = (0, 1)$, $p_3 = (1, 0)$, $p_4 = (2, 0)$, as shown in dark gray in the figure.
    We have the same free disposal hull as in \cref{ex:1}, which is a convex set, and thus we can apply \cref{thm:main} to this case.
    
    \begin{figure}[ht]
        \includegraphics[width=.66\textwidth]{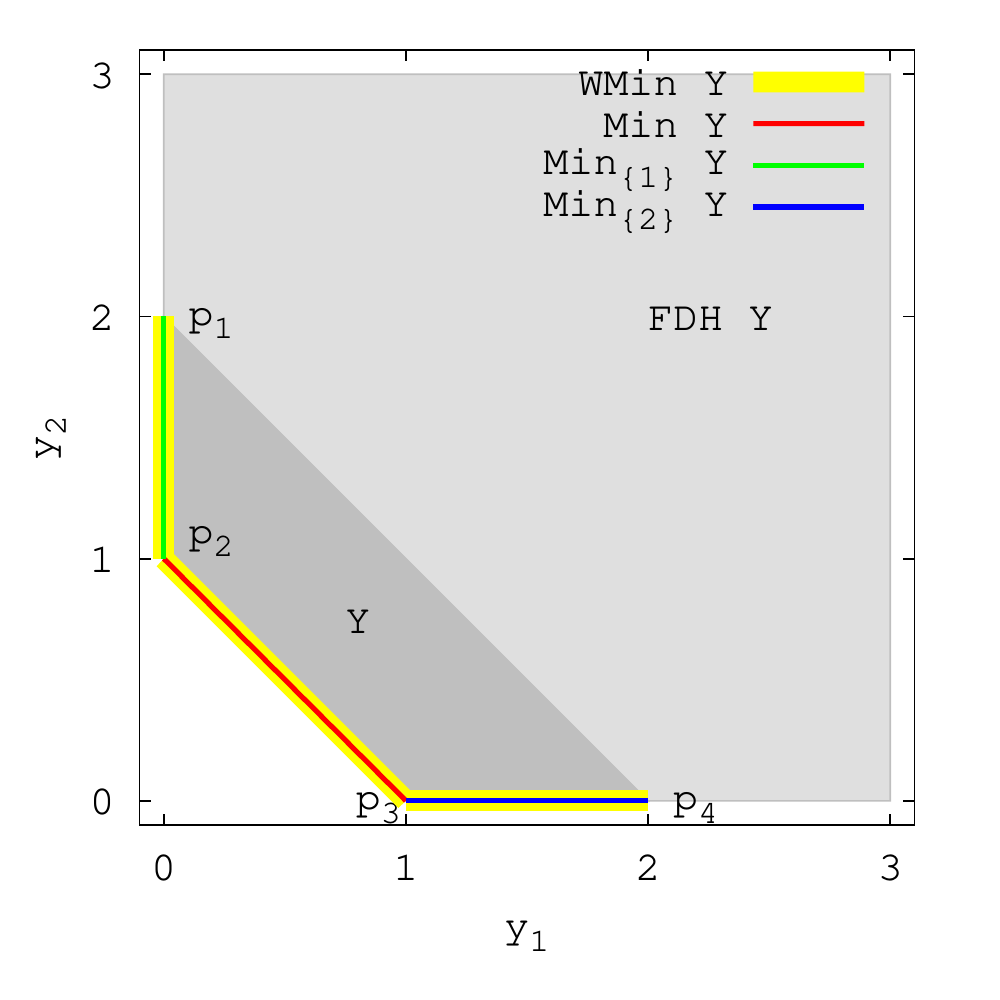}%
        \caption{The condition $\lnot (\beta)$ implies $\lnot (\alpha)$ on convex $\FDH Y$.}\label{fig:example2}
    \end{figure}
    
    We can easily check:
    \begin{align*}
        \M_{\Set{1}} Y & = \overline{p_1 p_2},\\
        \M_{\Set{2}} Y & = \overline{p_3 p_4},\\
        \M           Y & = \overline{p_2 p_3},\\
        \WM          Y & = \overline{p_1 p_2} \cup \overline{p_2 p_3} \cup \overline{p_3 p_4}.
    \end{align*}
    Since $\M_{\Set{1}} Y \not \subseteq \M Y$, the condition $(\beta)$ in \cref{thm:main} does not hold.
    By \cref{thm:main}, the condition $(\alpha)$ $\WM Y = \M Y$ does not hold, as shown in the above equations.
\end{example}

The following example shows why the assumption of \cref{thm:main} is required.
\begin{example}\label{ex:3}
    Let us consider the situation shown in \cref{fig:example3} where the domain $Y$ is the nonconvex polygon with five vertices $p_1 = (0, 3)$, $p_2 = (1, 2)$, $p_3 = (1, 1)$, $p_4 = (2, 0)$, $p_5 = (2, 3)$, shown in dark gray.
    
    \begin{figure}[ht]
        \includegraphics[width=.66\textwidth]{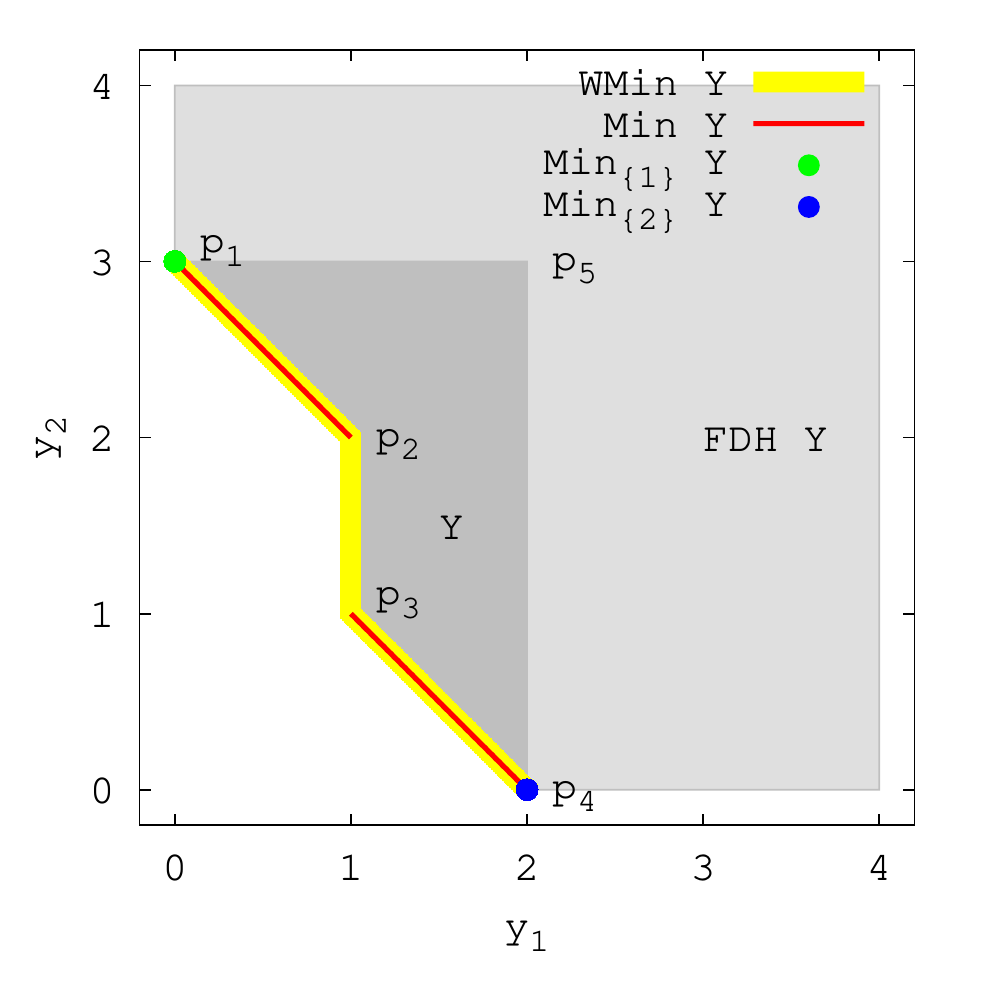}%
        \caption{The condition $(\beta)$ holds, but $(\alpha)$ does not on nonconvex $\FDH Y$ .}\label{fig:example3}
    \end{figure}
    
    We can easily check:
    \begin{align*}
        \M_{\Set{1}} Y & = \Set{p_1},\\
        \M_{\Set{2}} Y & = \Set{p_4},\\
        \M           Y & = \overline{p_1 p_2} \cup \overline{p_3 p_4} \setminus \Set{p_2},\\
        \WM          Y & = \overline{p_1 p_2} \cup \overline{p_2 p_3} \cup \overline{p_3 p_4}.
    \end{align*}
    Since $\M_{\Set{1}} Y \subseteq \M Y$ and $\M_{\Set{2}} Y \subseteq \M Y$, the condition $(\beta)$ holds.
    However, the free disposal hull of $Y$ is a nonconvex set, as shown in light gray in the figure.
    Hence, we cannot apply \cref{thm:main} to this case.
    In such a case, $(\alpha)$ can be false even if $(\beta)$ is true.
    Actually, in this example, the condition $(\alpha)$ does not hold as seen in the above equations.
\end{example}

In \cref{thm:main}, the assumption (the free disposal hull of $Y$ is convex) is not a necessary condition.
In the following example, we will give a case where the free disposal hull is nonconvex but the condition $(\alpha)$ holds (thus, $(\beta)$ also holds).

\begin{example}\label{ex:4}
    Let us consider the situation shown in \cref{fig:example4} where the domain $Y$ is a nonconvex polygon with four vertices $p_1 = (0, 3)$, $p_2 = (2, 2)$, $p_3 = (3, 0)$, $p_4 = (3, 3)$, as shown in dark gray in the figure.
    
    \begin{figure}[ht]
        \includegraphics[width=.66\textwidth]{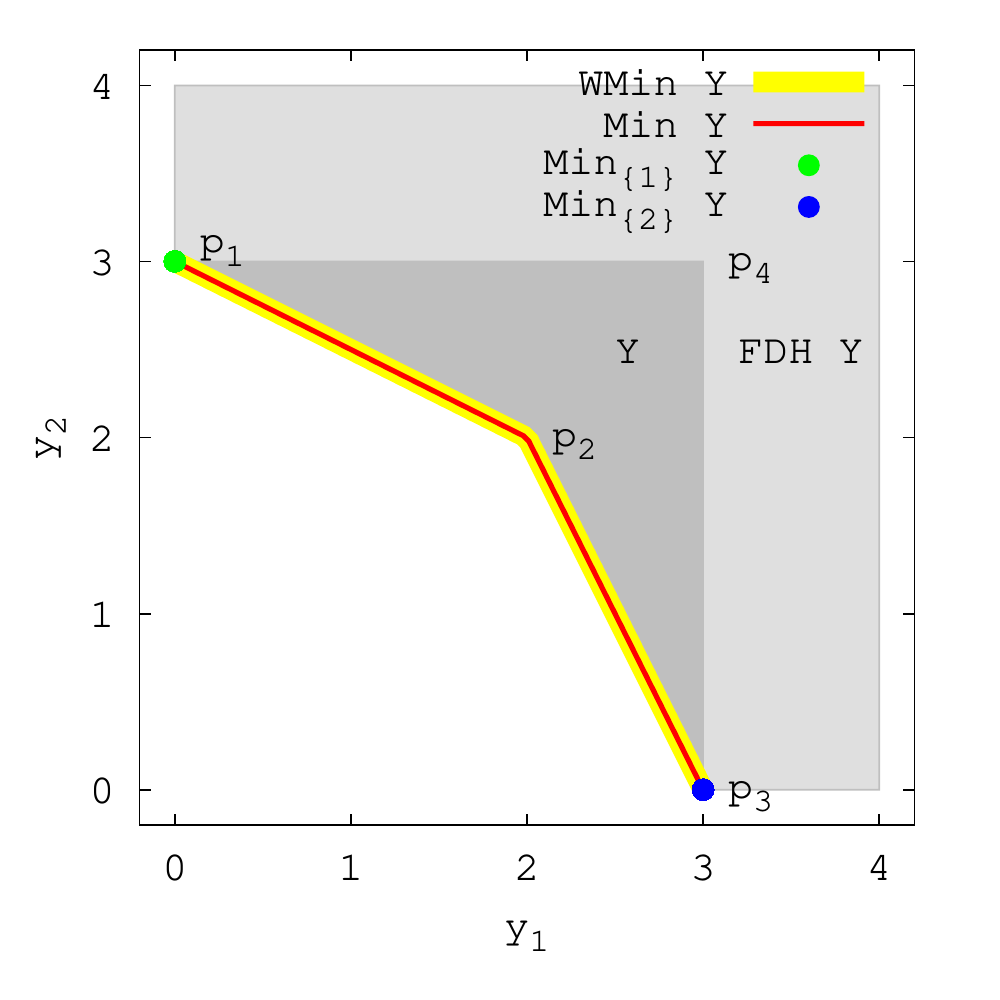}%
        \caption{The condition $(\beta)$ holds, and $(\alpha)$ does on nonconvex $\FDH Y$.}\label{fig:example4}
    \end{figure}
    
    We can easily check:
    \begin{align*}
        \M_{\Set{1}} Y & = \Set{p_1},\\
        \M_{\Set{2}} Y & = \Set{p_3},\\
        \WM Y = \M   Y & = \overline{p_1 p_2} \cup \overline{p_2 p_3}.
    \end{align*}
    Since $\M_{\Set{1}} Y \subseteq \M Y$ and $\M_{\Set{2}} Y \subseteq \M Y$, the condition $(\beta)$ holds.
    The free disposal hull of $Y$ is a nonconvex set, as shown in light gray in the figure.
    Hence we cannot apply \cref{thm:main} to this case.
    Nevertheless, the condition $(\alpha)$ actually holds as seen in the above equations.
\end{example}

%%%%%%%%%%%%%%%%%%%%%%%%%%%%%%%%%%%%%%%%%%%%%%%%%%%%%%%%%%%%%%%%%%%%%%%%%%%%%%%%
\section{Proof of \texorpdfstring{\cref{thm:main}}{Theorem 2.1}}\label{sec:main}
In the case $Y=\emptyset$, it is trivially seen that both $(\alpha)$ and $(\beta)$ hold.
Hence, in what follows, we will consider the case $Y\not=\emptyset$.

\subsection{Proof of \texorpdfstring{$(\alpha) \Rightarrow (\beta)$}{(a) => (b)}}\label{subsec:main_proof_1}
Let $I$ be a nonempty subset of $M$.
Then, it is clearly seen that
\begin{equation*}
    \M_I Y \subseteq \WM_I Y \subseteq \WM Y.
\end{equation*}
By $(\alpha)$, we get $\M_I Y \subseteq \M Y$.
Thus, we have $(\beta)$.
\QED

\subsection{Proof of \texorpdfstring{$(\beta) \Rightarrow (\alpha)$}{(b) => (a)}}\label{subsec:main_proof_2}
It is sufficient to show that $\WM Y \subseteq \M Y$.
Let $y^*=(y^*_1\ld y^*_m) \in \WM Y$ be an arbitrary element.
Set
\begin{align*}
    A & = \Set{(y_1-y^*_1\ld y_m-y^*_m)\in \R^m | (y_1\ld y_m)\in \FDH Y},\\
    B & = \Set{y \in \R^m | y <_M 0}.
\end{align*}
Here, note that $0=(0\ld 0)\in \R^m$ in the above description of $B$.
Then, we will have $A \cap B = \emptyset$ by contradiction.
Suppose that $A \cap B \neq \emptyset$.
Then, there exist $y' \in Y$ and $z \in \R^m_{\geq 0}$ satisfying $y' + z - y^* <_M 0$.
Since $z \in \R^m_{\geq 0}$, we get $y' - y^* <_M 0$.
This contradicts $y^* \in \WM Y$.
Hence, we have $A \cap B = \emptyset$.

In the following lemma, $\ang{,}$ stands for the inner product in $\R^m$.
\begin{lemma}[Separation theorem~\cite{Matousek2002}]\label{thm_convex}
Let $D_1$ and $D_2$ be nonempty convex subsets of $\R^m$ satisfying $D_1 \cap D_2 = \emptyset$.
Then, there exist $a = (a_1 \ld a_m) \in \R^m$ $(a \neq 0)$ and $b \in \R$ such that the following both assertions hold.
\begin{enumerate}[$(1)$]
    \item For any $y \in D_1$, we have $\ang{a, y} \geq b$.
    \item For any $y \in D_2$, we have $\ang{a, y} \leq b$.
\end{enumerate}
\end{lemma}
Note that $A$ and $B$ are nonempty convex subsets of $\R^m$.
Hence, by \cref{thm_convex}, there exist $a = (a_1 \ld a_m) \in \R^m$ $(a \neq 0)$ and $b \in \R$ such that the following both assertions hold.
\begin{enumerate}[(1')]
    \item  For any $y \in A$, we have $\ang{a, y} \geq b$.
    \item  For any $y \in B$, we have $\ang{a, y} \leq b$.
\end{enumerate}
Then, we will show that $b = 0$.
Since $0 = (0 \ld 0) \in A$, we have $\ang{a, 0} \geq b$ by (1').
Namely, we get $b \leq 0$.
Since $(-\ep \ld -\ep) \in B$ for any sufficiently small $\ep > 0$, it is clearly seen that $b = 0$ by (2').

We will show that $a_i \geq 0$ for any $i \in M$ by contradiction.
Suppose that there exists an element $i' \in M$ satisfying $a_{i'} < 0$.
Let $y = (y_1 \ld y_m) \in \R^m$ be the element given by
\begin{equation*}
y_i = \begin{cases}
-1 & \text{if $i \neq i'$},\\
\dfrac{\prn{\sum_{j = 1, j \neq i'}^m \abs{a_j}} + 1}{a_{i'}} & \text{if $i = i'$}.
\end{cases}
\end{equation*}
Then, we get $y \in B$ and $\ang{a, y} > 0$.
This contradicts (2').
Hence, it follows that $a_i \geq 0$ for any $i \in M$.

Now, set
\begin{equation*}
    I = \Set{ i \in M | a_i > 0}.
\end{equation*}
Notice that $I \neq \emptyset$.
Set $I = \Set{i_1 \ld i_k}$, where $k$ is an integer $(1 \leq k \leq m)$ and $i_1<\cdots <i_k$.

We will show that $y^*\in \M_I Y$.
Let $y = (y_1 \ld y_m) \in Y$ be any element.
Since $(y_1-y^*_1\ld y_m-y^*_m) \in A$ and $b = 0$, by (1'), we have
\begin{align*}
    a_{i_1}(y_{i_1} - y^*_{i_1}) + \cdots + a_{i_k}(y_{i_k} - y^*_{i_k}) \geq 0.
\end{align*}
Since $a_{i_1} > 0 \ld a_{i_k} > 0$, the element $y \in Y$ does not satisfy $y \lneq_I y^*$.
Therefore, we obtain $y^* \in \M_I Y$.
By the assumption $(\beta)$, it follows that $y^* \in \M Y$.
\QED

%%%%%%%%%%%%%%%%%%%%%%%%%%%%%%%%%%%%%%%%%%%%%%%%%%%%%%%%%%%%%%%%%%%%%%%%%%%%%%%%
\section{Preliminaries for applications of \texorpdfstring{\cref{thm:main_app}}{Proposition 2.3}}\label{sec:lemma}
Let $X$ be a convex subset of $\R^n$.
A function $f: X \to \R$ is said to be \emph{convex} if
\begin{equation*}
    f(t x + (1 - t) y) \leq t f(x) + (1 - t) f(y)
\end{equation*}
for all $x, y \in X$ and all $t \in [0, 1]$.
A function $f: X \to \R$ is said to be \emph{strongly convex} if there exists $\alpha > 0$ satisfying
\begin{equation*}
    f(t x + (1 - t) y) \leq t f(x) + (1 - t) f(y) - \frac{1}{2} \alpha t (1 - t) \norm{x - y}^2.
\end{equation*}
for all $x, y \in X$ and all $t \in [0, 1]$, where $\norm{x - y}$ denotes the Euclidean norm of $x - y$.
For details on convex functions and strongly convex functions, see \cite{Nesterov2004}.
A mapping $f = (f_1 \ld f_m): X \to \R^m$ is said to be \emph{convex} (resp., \emph{strongly convex}) if every $f_i$ is convex (resp., strongly convex).

First, we give the following well-known result.
For the sake of the readers' convenience, we also give the proof.
\begin{lemma}\label{thm:app_convex_lemma}
Let $X$ be a convex subset of $\R^n$, and $f: X \to \R^m$ be a convex mapping.
Then, the free disposal hull of $f(X)$ is convex.
\end{lemma}
\begin{proof}[Proof of \cref{thm:app_convex_lemma}]
Let $y=(y_1\ld y_m)$,  $\wt{y}=(\wt{y}_1\ld \wt{y}_m)\in \FDH f(X)$ be arbitrary points and $t\in [0,1]$ be an arbitrary element.
Then, there exist $x\in X$ (resp., $\wt{x}\in X$) and $z = (z_1\ld z_m) \in \R^m_{\geq 0}$ (resp., $\wt{z}=(\wt{z}_1\ld \wt{z}_m)\in \R^m_{\geq 0})$ such that $y=f(x)+z$ (resp., $\wt{y}=f(\wt{x})+\wt{z}$).
Let $i$ be an arbitrary integer satisfying $1\leq i\leq m$.
Since $f_i$ is convex, we have
\begin{align*}
    f_i(tx+(1-t)\wt{x}) \leq tf_i(x)+(1-t)f_i(\wt{x}),
\end{align*}
where $f=(f_1\ld f_m)$.
Since $z_i\geq 0$ and $\wt{z}_i\geq 0$ for any $i=1\ld m$, we also get
\begin{align*}
    tf_i(x)+(1-t)f_i(\wt{x})\leq t(f_i(x)+z_i)+(1-t)(f_i(\wt{x})+\wt{z}_i) = ty_i+(1-t)\wt{y}_i.
\end{align*}
Hence, we obtain
\begin{align}\label{eq:app_conv_1}
    f_i(tx+(1-t)\wt{x}) \leq ty_i+(1-t)\wt{y}_i.
\end{align}
Since we have \cref{eq:app_conv_1} for any integer $i$ satisfying $1\leq i\leq m$, it follows that $ty + (1 - t) \wt{y} \in \FDH f(X)$.
\end{proof}

In the following, for two sets $U, V$, and a subset $W$ of $U$, the restriction of a given mapping $g: U \to V$ to $W$ is denoted by $g|_{W}: W \to V$.

\begin{lemma}\label{thm:lem_injective}
Let $f: X \to \R^m$ be a mapping, where $X$ is a given set.
Let $I=\Set{i_1\ld i_k}$ $(i_1<\cdots <i_k)$ be a nonempty subset of $M$, where $k$ is the number of the elements of $I$.
If $f_I|_{\E(f_I, X)}: \E(f_I, X) \to \R^k$ is injective, then we have $\E(f_I, X) \subseteq \E(f, X)$.
\end{lemma}
\begin{proof}[Proof of \cref{thm:lem_injective}]
Suppose that there exists an element $x \in \E(f_I, X)$ such that $x \not \in \E(f, X)$.
Then, there exists an element $y \in X$ ($y \neq x$) satisfying $f_i(y) \leq f_i(x)$ for any $i \in M$.
Since $I \subseteq M$, it follows that $f_i(y) \leq f_i(x)$ for any $i \in I$.
Since $x \in \E(f_I, X)$, we get $f_I(x) = f_I(y)$.
Therefore, we have $y \in \E(f_I, X)$.
This contradicts the assumption that $f_I|_{\E(f_I, X)}: \E(f_I, X) \to \R^k$ is injective.
\end{proof}

\begin{lemma}\label{thm:lem_injective_strongly}
Let $X$ be a convex subset of $\R^n$, and $f: X \to \R^m$ be a strongly convex mapping.
Then, $f|_{\E(f, X)}: \E(f, X) \to \R^m$ is injective.
\end{lemma}
\begin{proof}[Proof of \cref{thm:lem_injective_strongly}]
Suppose that $f|_{\E(f, X)}: \E(f, X) \to \R^m$ is not injective.
Then, there exist $x, y \in \E(f, X)$ such that $x \neq y$ and $f|_{\E(f, X)}(x) = f|_{\E(f, X)}(y)$.
Let $i$ be an arbitrary integer satisfying $1\leq i\leq m$.
Since $f = (f_1 \ld f_m)$ is strongly convex, there exists $\alpha_i > 0$ satisfying
\begin{equation*}
    f_i(tx + (1 - t)y) \leq tf_i(x) + (1 - t)f_i(y) - \frac{1}{2}\alpha_i t(1 - t) \norm{x - y}^2
\end{equation*}
for the points $x, y \in \E(f, X)$ and all $t \in [0, 1]$.
Set $t = \frac{1}{2}$.
Then, we get
\begin{equation*}
    f_i \prn{\frac{x + y}{2}} \leq \frac{f_i(x) + f_i(y)}{2} - \frac{\alpha_i}{8} \norm{x - y}^2.
\end{equation*}
Since $f|_{\E(f, X)}(x) = f|_{\E(f, X)}(y)$, we have
\begin{equation*}
    f_i \prn{\frac{x + y}{2}} \leq f_i(x) - \frac{\alpha_i}{8} \norm{x - y}^2.
\end{equation*}
Since $x \neq y$ and $\alpha_i > 0$, it follows that
\begin{equation*}
    f_i \prn{\frac{x + y}{2}} < f_i(x).
\end{equation*}
This contradicts $x \in \E(f, X)$.
\end{proof}

%%%%%%%%%%%%%%%%%%%%%%%%%%%%%%%%%%%%%%%%%%%%%%%%%%%%%%%%%%%%%%%%%%%%%%%%%%%%%%%%
\section{Mathematical applications of \texorpdfstring{\cref{thm:main_app}}{Proposition 2.3}}\label{sec:application}
In this section, as mathematical applications of \cref{thm:main_app}, we give  \cref{thm:app_imi,thm:app_convex_mapping,thm:app_convex,thm:app_strong}.

First, \cref{thm:main_app} gives a characterization of the equality of the weak efficiency and the efficiency for possibly nonconvex problems having a convex image $f(X)$ as follows:
\begin{corollary}\label{thm:app_imi}
Let $f = (f_1 \ld f_m): X \to \R^m$ be a mapping, where $X$ is a given set.
If $f(X)$ is convex, then the following $(\alpha)$ and $(\beta)$ are equivalent: \begin{enumerate}
    \item[$(\alpha)$]
        $\WE(f, X) = \E(f, X)$.
    \item[$(\beta)$]
        ${\displaystyle \bigcup_{\emptyset \neq I \subseteq M} \E(f_I, X) \subseteq \E(f, X)}$.
\end{enumerate}
\end{corollary}
\begin{proof}[Proof of \cref{thm:app_imi}]
Since $f(X)$ is convex, it is clearly seen that the free disposal hull of $f(X)$ is also convex.
Thus, by \cref{thm:main_app}, we have \cref{thm:app_imi}.
\end{proof}

Unfortunately, it is not easy to check the convexity of the image of a given mapping which is possibly nonconvex.
A more workable condition ensuring this characterization is the convexity of a given mapping itself.
\begin{corollary}\label{thm:app_convex_mapping}
Let $X$ be a convex subset of $\R^n$ and $f = (f_1 \ld f_m): X \to \R^m$ be a convex mapping.
Then, the following $(\alpha)$ and $(\beta)$ are equivalent:
\begin{enumerate}
    \item[$(\alpha)$]
        $\WE(f, X) = \E(f, X)$.
    \item[$(\beta)$]
        ${\displaystyle \bigcup_{\emptyset \neq I \subseteq M} \E(f_I, X) \subseteq \E(f, X)}$.
\end{enumerate}
\end{corollary}
\begin{proof}[Proof of \cref{thm:app_convex_mapping}]
Since $f$ is convex, by \cref{thm:app_convex_lemma}, the free disposal full of $f(X)$ is convex.
Therefore, by \cref{thm:main_app}, we get \cref{thm:app_convex_mapping}.
\end{proof}
As a direct consequence, the characterization is valid for convex programming problems.
\begin{corollary}\label{thm:app_convex}
Let $X$ be a convex subset of $\R^n$ and $f = (f_1 \ld f_m): X \to \R^m$ be a convex mapping.
Let $g_1 \ld g_\ell$ be convex functions of $X$ into $\R$, where $\ell$ is a positive integer.
Set
\begin{equation*}
    \Omega = \Set{x \in X | g_1(x) \leq 0 \ld g_\ell (x) \leq 0}.
\end{equation*}
Then, the following $(\alpha)$ and $(\beta)$ are equivalent:
\begin{enumerate}
    \item[$(\alpha)$]
        $\WE(f|_\Omega, \Omega) = \E(f|_\Omega, \Omega)$.
    \item[$(\beta)$]
        ${\displaystyle \bigcup_{\emptyset \neq I \subseteq M} \E((f|_\Omega)_I, \Omega) \subseteq \E(f|_\Omega, \Omega)}$.
\end{enumerate}
\end{corollary}
\begin{proof}[Proof of \cref{thm:app_convex}]
Since $g_1 \ld g_\ell$ are convex functions, it is clearly seen that $\Omega$ is convex.
Since the mapping $f|_\Omega: \Omega \to \R^m$ is convex, by  \cref{thm:app_convex_mapping}, we get \cref{thm:app_convex}.
\end{proof}

Let $f = (f_1 \ld f_m): X \to \R^m$ be a mapping, where $X$ is a set.
Then, $x^* \in X$ is called a \emph{strictly efficient solution} if there does not exist $x \in X$ $(x \neq x^*)$ such that $f_i(x) \leq f_i(x^*)$ for all $i = 1 \ld m$.
We denote the set of all strictly efficient solutions to the problem minimizing $f$ by $\SE(f, X)$.

\begin{corollary}\label{thm:app_strong}
    Let $X$ be a convex subset of $\R^n$, and $f: X \to \R^m$ be a strongly convex mapping.
    Then, we have
    \begin{align*}
        \WE(f, X) = \E(f, X) = \SE(f, X).
    \end{align*}
\end{corollary}
\begin{proof}[Proof of \cref{thm:app_strong}]
Since $f|_{\E(f, X)}$ is injective by \cref{thm:lem_injective_strongly}, it is not hard to see that $\E(f, X) = \SE(f, X)$.

Now, we will show that $\WE(f, X) = \E(f, X)$.
Since $f$ is strongly convex, by \cref{thm:app_convex_lemma}, the set $\FDH f(X)$ is convex.
Thus, by \cref{thm:main_app}, in order to show that $\WE(f, X) = \E(f, X)$, it is sufficient to show
\begin{equation}\label{eq:I}
    \bigcup_{\emptyset \neq I \subseteq M} \E(f_I, X) \subseteq \E(f, X).
\end{equation}
Let $I$ be any nonempty subset of $M$.
Since $f _I|_{\E(f_I, X)}: \E(f_I, X) \to \R^{k}$ is strongly convex, by \cref{thm:lem_injective_strongly}, the mapping $f _I|_{\E(f_I, X)}$ is injective, where $k$ is the number of the elements of $I$.
By \cref{thm:lem_injective}, we have $\E(f_I, X) \subseteq \E(f, X)$.
Thus, we obtain \cref{eq:I}.
\end{proof}

%%%%%%%%%%%%%%%%%%%%%%%%%%%%%%%%%%%%%%%%%%%%%%%%%%%%%%%%%%%%%%%%%%%%%%%%%%%%%%%%
\section{A practical application}\label{sec:practical-applications}
In this section, as a practical application of \cref{thm:main_app}, we show that all weakly efficient solutions to a multi-objective LASSO with mild modification are efficient.
The LASSO is a sparse modeling method that is originally proposed as a single-objective optimization problem~\cite{Tibshirani1996} and sometimes treated as a multi-objective one~\cite{Coelho2020}.
Let us consider a linear regression model:
\begin{align*}
    y = \theta_1 x_1 + \theta_2 x_2 + \dots + \theta_n x_n + \xi,
\end{align*}
where $x_i$ and $\theta_i$ $(i=1\ld n)$ are a predictor and its coefficient, $y$ is a response to be predicted, and $\xi$ is a Gaussian noise.
Given a matrix $X$ with $m$ rows of observations and $n$ columns of predictors and a row vector $y$ of $m$ responses, the (original) LASSO regressor is the solution to the following problem:
\begin{equation}\label{eqn:lasso}
    \minimize_{\theta \in \R^n} g_{\lambda}(\theta) := \frac{1}{2m} \norm{X \theta - y}^2 + \lambda \abs{\theta},
\end{equation}
where $\abs{\;\cdot\;}$ is the $\ell_1$-norm and $\lambda$ is a user-specified positive number to force the solution to be sparse (i.e., the optimal $\theta$ contains many zeros).
Note that with $\lambda = 0$, the problem \cref{eqn:lasso} reduces to the ordinary least squares (OLS) regression.
Choosing an appropriate value for $\lambda$ requires repeated solution of \cref{eqn:lasso} with varying $\lambda$, which is the most time-consuming part of this method.

To find a good solution without such a costful hyper-parameter search, the problem \cref{eqn:lasso} is sometimes reformulated as a multi-objective one whose efficient solutions are optimal solutions to the original problem \cref{eqn:lasso} with different $\lambda$'s (for example, see \cite{Coelho2020}).
We treat the OLS term and the regularization term as individual objective functions:
\[
    f_1(\theta) = \frac{1}{2m} \norm{X \theta - y}^2,\quad
    f_2(\theta) = \abs{\theta}.
\]
In the multi-objective problem of minimizing $f=(f_1,f_2)$, the equality $\WE(f, \R^n)=\E(f, \R^n)$ does not necessarily hold (for example, if $X$ is a zero matrix, then $\WE(f, \R^n)=\R^n$ and $\E(f, \R^n)=\set{(0\ld0)}$).

To avoid such a corner case, we consider a modified version of multi-objective LASSO:
\begin{equation}\label{eqn:lasso-mop}
    \begin{split}
        \minimize_{\theta \in \R^n}\ & \wt{f}(\theta) := (\wt{f}_1(\theta), \wt{ f}_2(\theta))\\
        \text{where }                & \wt{f}_i(\theta) = f_i(\theta) + \ep f_2(\theta) \quad (i = 1, 2).
    \end{split}
\end{equation}
In \cref{eqn:lasso-mop}, we assume that $\ep$ is a positive real number.
Note that $\wt{f}$ in \cref{eqn:lasso-mop} is a non-differentiable mapping that is convex but never strongly convex.
As a practical application of \cref{thm:main_app}, we can show the following.
\begin{theorem}\label{thm:lasso}
In \cref{eqn:lasso-mop}, we have 
$\E(\wt{f}, \R^n)=\WE(\wt{f}, \R^n)$.
\end{theorem}
\begin{proof}[Proof of \cref{thm:lasso}]

Since $\wt{f}$ is convex, it is sufficient to show that $\E(\wt{f_i}, \R^n)\subseteq \E(\wt{f}, \R^n)$ for $i=1,2$ by \cref{thm:app_convex_mapping}, which is one of the applications of \cref{thm:main_app}.
Since $\wt{f_2}$ has the unique minimizer $(0\ld 0)\in \R^n$, we have $\E(\wt{f_2}, \R^n)\subseteq \E(\wt{f}, \R^n)$.
In order to show $\E(\wt{f_1}, \R^n)\subseteq \E(\wt{f}, \R^n)$, we prepare the following.
\begin{lemma}\label{thm:norm_lem}
Let $\theta^\ast, \wt{\theta}^\ast$ be elements of $\E(\wt{f_1}, \R^n)$.
Then, we have 
\begin{align*}
    \norm{X \theta^\ast - y}^2=\|X \wt{\theta}^\ast - y\|^2 \mbox{ and }
    \abs{\theta^\ast}=|\wt{\theta}^\ast|.
\end{align*}
\end{lemma}
\begin{proof}[Proof of \cref{thm:norm_lem}]
It is sufficient to consider the case $\theta^\ast\not=\wt{\theta}^\ast$.
Since $\wt{f_1}$ is convex, there exists a mapping $c=(c_1\ld c_n):[0,1]\to \R^n$ given by 
\begin{align*}
    c(t)=(a_1t+b_1\ld a_nt+b_n),
    \ \ c(0)=\theta^*, \ \ c(1)=\wt{\theta}^*
\end{align*}
such that $c(t)$ is the minimizer of $\wt{f}_1$ for all $t\in [0,1]$, where $a_1\ld a_n,b_1\ld b_n\in \R$ and $a_i\not=0$ for at least one index $i$.
Set 
\begin{align*}
    N&=\set{1\ld n},
    \\
    \wt{N}&=\set{i\in N|\mbox{$a_i=0$ and $b_i=0$}}.
\end{align*}
Since $\wt{N}\subsetneq N$, we set $N\setminus \wt{N}=\set{i_1\ld i_\rho}$.
Since $c$ is continuous, it is not hard to see that there exists an open interval $I$ $(\subseteq [0,1])$ such that for any $k=1\ld \rho$, either one of the following two holds:
\begin{enumerate}
    \item 
    We have $c_{i_k}(t)>0$ for all $t\in I$.
    \item
    We have $c_{i_k}(t)<0$ for all $t\in I$.
\end{enumerate}
Set $X=(x_{ij})_{1\leq i\leq m,1\leq j\leq n}$ and $y=(y_1\ld y_m)$.
Then, the composition $\wt{f}_1\circ c:I\to \R$ is expressed by 
\begin{align*}
    (\wt{f}_1\circ c)(t)
    =\frac{1}{2m}\sum_{j=1}^m
    \prn{
    \sum_{k=1}^\rho x_{j,i_k}(a_{i_k}t+b_{i_k})-y_j
    }^2
    +
    \ep \sum_{k=1}^\rho  q_{i_k}(a_{i_k}t+b_{i_k}),
\end{align*}
where 
\begin{align*}
    q_{i_k}&=
    \begin{cases}
    1 & \mbox{ if $c_{i_k}(t)>0$ for all $t\in I$}, \\
    -1 & \mbox{ if $c_{i_k}(t)<0$ for all $t\in I$}.
    \end{cases}
\end{align*}
By calculation, we obtain 
\begin{align*}
    \frac{d^2(\wt{f}_1\circ c)}{dt^2}(t)=\frac{1}{m}\sum_{j=1}^m\prn{
    \sum_{k=1}^\rho x_{j,i_k}a_{i_k}
    }^2,
\end{align*}
where $t\in I$.
Since $\wt{f}_1\circ c:[0,1]\to \R$ is a constant function, we have 
\begin{align}\label{eq:zero}
    \sum_{k=1}^\rho x_{j,i_k}a_{i_k}=0
\end{align}
for all $j=1\ld m$.

For simplicity, set $h(\theta)=\norm{X \theta - y}^2$.
Then, we have 
\begin{align*}
    (h\circ c)(t)&=\sum_{j=1}^m
    \prn{
    \sum_{k=1}^\rho x_{j,i_k}\prn{a_{i_k}t+b_{i_k}}-y_j
    }^2
    \\
    &=\sum_{j=1}^m
    \prn{
    \prn{\sum_{k=1}^\rho x_{j,i_k}a_{i_k}}t+\sum_{k=1}^\rho x_{j,i_k}b_{i_k}-y_j
    }^2
    \\
    &=\sum_{j=1}^m
    \prn{\sum_{k=1}^\rho x_{j,i_k}b_{i_k}-y_j
    }^2
\end{align*}
for all $t\in [0,1]$.
The last equality above is obtained by \cref{eq:zero}. 
Hence, since $h\circ c:[0,1]\to \R$ is a constant function, we obtain
\begin{align*}
    h(\theta^\ast)=(h\circ c)(0)=(h\circ c)(1)=h(\wt{\theta}^\ast).
\end{align*}
Since $\wt{f}_1(\theta^\ast)=\wt{f}_1(\wt{\theta}^\ast)$, we also have $|\theta^\ast|=|\wt{\theta}^\ast|$.
\end{proof}
Now, let $\theta\in \E(\wt{f}_1, \R^n)$ be an arbitrary element.
Suppose that $\theta\not\in\E(\wt{f}, \R^n)$.
Then, there exists $\wt{\theta}\in\R^n$ such that $\wt{f}_1(\wt{\theta})=\wt{f}_1(\theta)$ and $\wt{f}_2(\wt{\theta})<\wt{f}_2(\theta)$.
Thus, we have $\theta,\wt{\theta}\in \E(\wt{f}_1, \R^n)$ and $|\wt{\theta}|<|\theta|$.
This contradicts  \cref{thm:norm_lem}.
Hence, we obtain $\theta\in \E(\wt{f}, \R^n)$.
\end{proof}

%%%%%%%%%%%%%%%%%%%%%%%%%%%%%%%%%%%%%%%%%%%%%%%%%%%%%%%%%%%%%%%%%%%%%%%%%%%%%%%%
\section{Conclusion}\label{sec:conclusion}
In this paper, we have given a characterization of the equality of weak efficiency and efficiency when the free disposal hull of the domain is convex.
By this fact, we have presented four classes of optimization problems where this characterization holds, including convex problems.
As a practical application, we have also shown that all the weakly efficient solutions to a multi-objective LASSO defined by  \cref{eqn:lasso-mop} are efficient.
We expect that the scope of the multi-objective reformulation discussed in \cref{sec:practical-applications} is not limited to the LASSO.
The same idea may be applied to a wide range of sparse modeling methods, including the group lasso \cite{Yuan2006}, the fused lasso \cite{Tibshirani2005}, the graphical lasso \cite{Friedman2007}, the smooth lasso \cite{Hebiri2011}, the elastic net \cite{Zou2005}, etc.

%%%%%%%%%%%%%%%%%%%%%%%%%%%%%%%%%%%%%%%%%%%%%%%%%%%%%%%%%%%%%%%%%%%%%%%%%%%%%%%%
\section*{Acknowledgements}
The authors are grateful to Kenta Hayano, Yutaro Kabata, and Hiroshi Teramoto for their kind comments.
\mbox{Shunsuke~Ichiki} was supported by JSPS KAKENHI Grant Numbers JP19J00650 and JP17H06128.
This work is based on the discussions at 2018 IMI Joint Use Research Program, Short-term Joint Research ``multi-objective optimization and singularity theory: Classification of Pareto point singularities'' in Kyushu University.

\bibliographystyle{plain}
\bibliography{main}

\end{document}